\documentclass[11pt]{article}

\usepackage{fullpage,times,url,natbib}

\usepackage{amsthm,amsfonts,amsmath,amssymb,epsfig,color,float,graphicx,verbatim}
\usepackage{algorithm,algorithmic}

\usepackage{hyperref}
\hypersetup{
	colorlinks   = true, 
	urlcolor     = blue, 
	linkcolor    = blue, 
	citecolor   = black 
}

\newtheorem{theorem}{Theorem}
\newtheorem{proposition}{Proposition}
\newtheorem{lemma}{Lemma}

\newtheorem{remark}{Remark}

\newcommand{\reals}{\mathbb{R}}
\newcommand{\E}{\mathbb{E}}

\newcommand{\be}{\mathbf{e}}
\newcommand{\bx}{\mathbf{x}}
\newcommand{\bw}{\mathbf{w}}
\newcommand{\bg}{\mathbf{g}}

\newcommand{\bu}{\mathbf{u}}
\newcommand{\bv}{\mathbf{v}}

\newcommand{\br}{\mathbf{r}}

\newcommand{\by}{\mathbf{y}}

\newcommand{\Ocal}{\mathcal{O}}
\newcommand{\Acal}{\mathcal{A}}

\newcommand{\Fcal}{\mathcal{F}}

\newcommand{\Scal}{\mathcal{S}}

\newcommand{\norm}[1]{\|#1\|}

\newcommand{\secref}[1]{Sec.~\ref{#1}}
\newcommand{\subsecref}[1]{Subsection~\ref{#1}}
\newcommand{\figref}[1]{Fig.~\ref{#1}}
\renewcommand{\eqref}[1]{Eq.~(\ref{#1})}
\newcommand{\lemref}[1]{Lemma~\ref{#1}}

\newcommand{\thmref}[1]{Thm.~\ref{#1}}

\title{Can We Find Near-Approximately-Stationary Points\\ of Nonsmooth Nonconvex Functions?}
\author{Ohad Shamir\\Weizmann Institute of Science}
\date{}

\begin{document}

\maketitle

\begin{abstract}
	It is well-known that given a bounded, smooth nonconvex function, standard gradient-based methods can find $\epsilon$-stationary points (where the gradient norm is less than $\epsilon$) in $\mathcal{O}(1/\epsilon^2)$ iterations. However, many important nonconvex optimization problems, such as those associated with training modern neural networks, are inherently \emph{not} smooth, making these results inapplicable. Moreover, as recently pointed out in  \citet{zhang2020complexity}, it is generally impossible to provide finite-time guarantees for finding an $\epsilon$-stationary point of nonsmooth functions. Perhaps the most natural relaxation of this is to find points which are \emph{near} such $\epsilon$-stationary points. In this paper, we show that even this relaxed goal is hard to obtain in general, given only black-box access to the function values and gradients. We also discuss the pros and cons of  alternative approaches. 
\end{abstract}

\section{Introduction}

We consider optimization problems associated with functions $f:\reals^d\mapsto \reals$, where $f(\cdot)$ is globally Lipschitz and bounded from below, but otherwise satisfies no special structure -- in particular, it is not necessarily convex, and not necessarily differentiable everywhere. Clearly, in high dimensions, and for sufficiently complex $f(\cdot)$, it is generally impossible to efficiently find a global minimum. However, if we relax our goal to finding (approximate) \emph{stationary} points of $f(\cdot)$, then the nonconvexity is no longer an issue. In particular, it is known that if $f(\cdot)$ is \emph{smooth} -- namely, differentiable and with a Lipschitz gradient -- then for any $\epsilon>0$, simple gradient-based algorithms can find $\bx$ such that $\norm{\nabla f(\bx)}\leq \epsilon$, using only $\Ocal(1/\epsilon^2)$ gradient computations, independent of the dimension (see for example \cite{nesterov2012make,jin2017escape,carmon2019lower}). 

Unfortunately, many optimization problems of interest are inherently \emph{not} smooth. For example, when training modern neural networks, involving max operations and rectified linear units, the associated optimization problem is virtually always nonconvex as well nonsmooth. Thus, the positive results above, which crucially rely on smoothness, are inapplicable. Although there are positive results even for nonconvex nonsmooth functions, they tend to be either purely asymptotic in nature (e.g., \citet{benaim2005stochastic,kiwiel2007convergence,davis2018stochastic,majewski2018analysis}), or require additional structure which many problems of interest lack, such as weak convexity or some separation between nonconvex and nonsmooth components (e.g., \citet{duchi2018stochastic,davis2019stochastic,drusvyatskiy2019efficiency,bolte2018first,beck2020convergence}). This leads to the interesting question of developing black-box algorithms with non-asymptotic guarantees, for finding stationary points of general nonsmooth nonconvex functions.

In an elegant recent work, \citet{zhang2020complexity} raise this question, and provide several contributions. First, they point out that in a black-box model, where the algorithm accesses the function only by computing its values and gradients at various points, it is generally impossible in to find an approximately stationary point with finitely many queries, simply because the gradient can change abruptly and thus ``hide'' a stationary point inside some arbitrarily small neighborhood. Instead, they propose the following relaxation (based on the notion of $\delta$-differential introduced in  \citet{goldstein1977optimization}): Letting $\partial f(\bx)$ denote the gradient set\footnote{Under a standard generalization of gradients to nonsmooth functions -- see \subsecref{subsec:preliminaries}.} of $f(\cdot)$ at $\bx$, we say that a point $\bx$ is a \emph{$(\delta,\epsilon)$-stationary point}, if
\begin{equation}\label{eq:destat}
\min\{\norm{\bu}:\bu\in \text{conv}\{\cup_{\by:\norm{\by-\bx}\leq \delta}~\partial f(\by)\}\}~\leq~\epsilon~,
\end{equation}
where $\text{conv}\{\cdot\}$ is the convex hull. In words, there exists a convex combination of gradients at a $\delta$-neighborhood of $\bx$, whose norm is at most $\epsilon$. The authors then proceed to provide a dimension-free, gradient-based algorithm for finding $(\delta,\epsilon)$-stationary points, using $\Ocal(1/\delta\epsilon^3)$ queries, as well as study related settings. 

Although this constitutes a very useful algorithmic contribution to nonsmooth optimization, it is important to note that a $(\delta,\epsilon)$-stationary point $\bx$ (as defined above) \emph{does not} imply that $\bx$ is $\delta$-close to an $\epsilon$-stationary point of $f(\cdot)$, nor that $\bx$ necessarily resembles a stationary point. Intuitively, this is because the convex hull of the gradients might contain a small vector, without any of the gradients being particular small. This is formally demonstrated in the following proposition:

\begin{proposition}\label{prop:stat}
	For any $\delta>0$, there exists a differentiable function $f(\cdot)$ on $\reals^2$ which is $2\pi$-Lipschitz on a ball of radius $2\delta$ around the origin, and the origin is a $(\delta,0)$-stationary point, yet $\min_{\bx:\norm{\bx}\leq \delta}\norm{\nabla f(\bx)}\geq 1$.
\end{proposition}
\begin{figure}\label{fig:stat}
	\centering
	\includegraphics[trim=0cm 1.5cm 0cm 1.5cm,clip=true,width=0.7\linewidth]{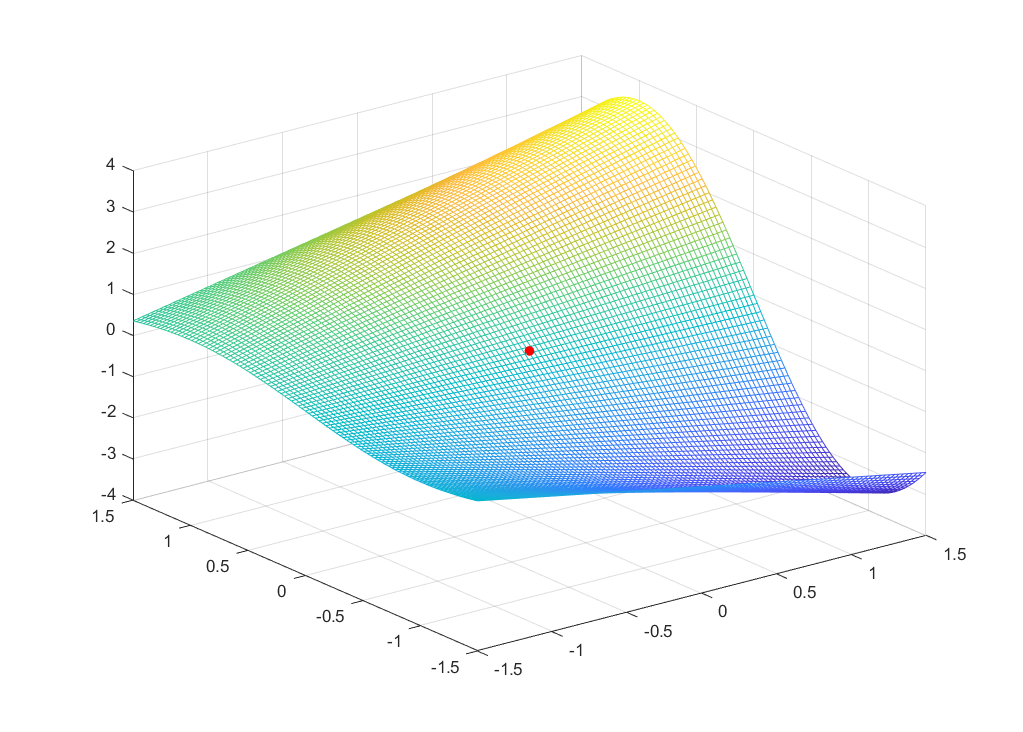}
	\caption{The function used in the proof of Proposition \ref{prop:stat}, for $\delta=1$. The origin (which fulfills the definition of a $(1,0)$-stationary point) is marked with a red dot. Best viewed in color.}
\end{figure}
\begin{proof}
	Fixing some $\delta>0$, consider the function
	\[
	f(u,v)~:=~(2\delta+u)\sin\left(\frac{\pi}{2\delta}v\right)
	\]
	(see \figref{fig:stat} for an illustration). This function is differentiable, and its gradient satisfies
	\[
	\nabla f(u,v)~=~ \left(\sin\left(\frac{\pi}{2\delta}v\right)~,~\frac{\pi}{2\delta}(2\delta+u)\cos\left(\frac{\pi}{2\delta}v\right)\right)~.
	\]
	First, we note that
	\[
	\frac{1}{2}\left(\nabla f(0,\delta)+\frac{1}{2}\nabla f(0,-\delta)\right)~=~
	\frac{1}{2}\left((1,0)+(-1,0)\right)~=~(0,0),
	\]
	which implies that $(0,0)$ is in the convex hull of the gradients at a distance at most $\delta$ from the origin, hence the origin is a $(\delta,0)$-stationary point. Second, we have that
	\begin{equation}\label{eq:nablaf2}
	\norm{\nabla f(u,v)}^2 = \sin^2\left(\frac{\pi}{2\delta}v\right)+\left(\frac{\pi}{2\delta}\right)^2(2\delta+u)^2\cos^2\left(\frac{\pi}{2\delta}v\right)~.
	\end{equation}
	For any $(u,v)$ of norm at most $2\delta$, we must have $|u|\leq 2\delta$, and therefore the above is at most
	\begin{align*}
	\sin^2\left(\frac{\pi}{2\delta}v\right)+\left(\frac{\pi}{2\delta}\right)^2(2\delta+2\delta)^2\cos^2\left(\frac{\pi}{2\delta}v\right)~\leq~ 4\pi^2\left(\sin^2\left(\frac{\pi}{2\delta}v\right)+\cos^2\left(\frac{\pi}{2\delta}v\right)\right)
	~=~ 4\pi^2~,
	\end{align*}
	which implies that the function is $2\pi$-Lipschitz on a ball of radius $2\delta$ around the origin. Finally, for any $(u,v)$ of norm at most $\delta$, we have $|u|\leq \delta$, so \eqref{eq:nablaf2} is at least
	\[
	\sin^2\left(\frac{\pi}{2\delta}v\right)+\left(\frac{\pi}{2\delta}\right)^2(2\delta-\delta)^2\cos^2\left(\frac{\pi}{2\delta}v\right)~\geq~ \sin^2\left(\frac{\pi}{2\delta}v\right)+\cos^2\left(\frac{\pi}{2\delta}v\right)~=~1~.
	\]
\end{proof}

\begin{remark}
Although the function $f(\cdot)$ in the proof has a constant Lipschitz parameter only close to the origin, it can be easily modified to be globally Lipschitz and bounded, for example by considering the function
\[
\tilde{f}(\bx) ~=~ \begin{cases} f(\bx)& \norm{\bx}\leq 2\delta\\ \max\left\{0,2-\frac{\norm{\bx}}{2\delta}\right\}\cdot f\left(\frac{2\delta}{\norm{\bx}}\bx\right) & \norm{\bx}>2\delta\end{cases}~,
\]
which is identical to $f(\cdot)$ in a ball of radius $2\delta$ around the origin, but decays to $0$ for larger $\bx$, and can be verified to be globally bounded and Lipschitz independent of $\delta$. 
\end{remark}

This result suggests that we should drop the $\text{conv}\{\cdot\}$ operator from the definition of $(\delta,\epsilon)$-stationarity in \eqref{eq:destat}, or equivalently, try to find \emph{near}-approximately-stationary points: Namely, getting $\delta$-close to a point $\bx$ such that $\partial f(\bx)$ contains an element with norm at most $\epsilon$. This is arguably the most natural way to relax the goal of finding $\epsilon$-stationary points, while hopefully still getting meaningful algorithmic guarantees. Unfortunately, we will show in the following section that this already sets the bar too high: For a very large class of gradient-based algorithms (and in fact, all of them under a mild assumption), it is impossible to find near-approximately-stationary point with worst-case finite-time guarantees, for small enough constant $\delta,\epsilon$. Thus, we cannot strengthen the notion of $(\delta,\epsilon)$-stationarity in this manner, and still hope to get similar algorithmic guarantees. In \secref{sec:discussion}, we further discuss the result and its implications.

\section{Hardness of Finding Near-Approximate-Stationary Points}

We begin by formalizing the setting in which we prove our hardness result (\subsecref{subsec:preliminaries}), followed by the main result in \subsecref{subsec:main} and its proof in \subsecref{subsec:proof}.

\subsection{Preliminaries}\label{subsec:preliminaries}

\textbf{Generalized Gradients and Stationary points.} First, we formalize the notion of gradients and stationary points for nonsmooth functions (which may not be everywhere differentiable). Given a Lipschitz function $f(\cdot)$ and a point $\bx$ in its domain, we let $\partial f(\bx)$ denote the set of \emph{generalized} gradients (following \citet{clarke1990optimization} and \citet{zhang2020complexity}), which is perhaps the most standard extension of the notion of gradients to nonsmooth nonconvex functions. For Lipschitz functions (which are almost everywhere differentiable by Rademacher's theorem), one simple way to define it is
\[
\partial f(\bx)~:=~ \text{conv}\{\bu:\bu=\lim_{k\rightarrow \infty} \nabla f(\bx_k), \bx_k\rightarrow \bx\}
\]
(namely, the convex hull of all limit points of $\nabla f(\bx_k)$, over all sequences $\bx_1,\bx_2,\ldots$ of differentiable points of $f(\cdot)$ which converge to $\bx$). With this definition, a \emph{(Clarke) stationary point} with respect to $f(\cdot)$ is a point $\bx$ satisfying $\mathbf{0}\in \partial f(\bx)$. Also, given some $\epsilon \geq 0$, we say that $\bx$ is an \emph{$\epsilon$-stationary point} with respect to $f(\cdot)$, if there is some $\bu\in \partial f(\bx)$ such that $\norm{\bu}\leq \epsilon$. To make the problem of getting near $\epsilon$-stationary points non-trivial, and following \citet{zhang2020complexity}, we will focus on functions $f(\cdot)$ that are both globally Lipschitz and bounded from below. In particular, we will assume that $f(\mathbf{0})-\inf_{\bx}f(\bx)$ is upper bounded by a constant (this is without loss of generality, as $\mathbf{0}$ can be replaced by any other fixed reference point). 

\textbf{Oracle Complexity.} We will study the algorithmic efficiency of the problem using the standard framework of (first-order) oracle complexity \citep{nemirovskyyudin1983}: Given a class of Lipschitz and bounded functions $\mathcal{F}$ as above, we associate with each $f\in \Fcal$ an \emph{oracle}, which for any $\bx$ in the domain of $f(\cdot)$, returns the value and a (generalized) gradient of $f(\cdot)$ at $\bx$. We focus on iterative algorithms which can be described via an interaction with such an oracle: At every iteration $t$, the algorithm chooses an iterate $\bx_t$, and receives from the oracle a generalized gradient and value of $f(\cdot)$ at $\bx_t$. The algorithm then uses the values and gradients obtained up to iteration $t$ to choose the point $\bx_{t+1}$ in the next iteration. This framework captures essentially all first-order algorithms for black-box optimization. In this framework, we fix some iteration budget $T$, and study the properties of the iterates $\bx_1,\ldots,\bx_T$ as a function of $T$ and the properties of $\Fcal$. 

\textbf{Algorithmic Families.} We will focus on two broad families of algorithms, which together span nearly all algorithms of interest: The first is the class of all \emph{deterministic} algorithms (denoted as $\Acal_{det}$), which are all oracle-based algorithms where where $\bx_1$ is chosen deterministically, and for all $t>1$, $\bx_{t}$ is some deterministic function of $\bx_1$ and the previously observed values and gradients. The second is the class of all \emph{linear-span} algorithms (denoted as $\Acal_{span}$), which are all deterministic or randomized oracle-based algorithms that initialize at some $\bx_1$ (which we will take to be $\mathbf{0}$ without loss of generality), and 
\[
\forall t>1,~~ \bx_t ~\in~ \text{span}(\bg_1,\ldots,\bg_{t-1})~,
\]
with $\bg_i$ being the gradient provided by the oracle at iteration $i$, given query point $\bx_i$. 

\subsection{Main Result}\label{subsec:main}

Our main result is the following:

\begin{theorem}\label{thm:main}
	There exist a large enough universal constant $C>0$ and a small enough universal constant $c>0$ such that the following holds: For any algorithm in $ \Acal_{det}\cup \Acal_{span}$, any $T>1$, and any $d\geq 2T$, there is a function $h(\cdot)$ on $\reals^{d}$ such that
	\begin{itemize}
		\item $h(\cdot)$ is $C$-Lipschitz, and $h(\mathbf{0})-\inf_{\bx}h(\bx)\leq C$.
		\item With probability at least $1-T\exp(-cd)$ over the algorithm's randomness (or deterministically if the algorithm is deterministic), the iterates $\bx_1,\ldots,\bx_T$ produced by the algorithm satisfy 
		\[
		\inf_{\bx\in\Scal}\min_{t\in \{1,\ldots,T\}}\norm{\bx_t-\bx}\geq c~,
		\]
		where $\Scal$ is the set of $c$-stationary points of $f(\cdot)$. 
	\end{itemize}
\end{theorem}

The theorem implies that for a very large family of algorithms, it is impossible to obtain worst-case, finite-time guarantees for finding near-approximately-stationary points of Lipschitz, bounded-from-below functions

Before continuing, we note that the result can be extended to \emph{any} oracle-based algorithm (randomized or deterministic) under a widely-believed assumption -- see remark \ref{remark:anyalg} in the proof for details. We also make two additional remarks:

\begin{remark}[More assumptions on $h(\cdot)$]
	The Lipschitz functions $h(\cdot)$ used to prove the theorem are based on a simple composition of affine functions, the Euclidean norm function $\bx\mapsto \norm{\bx}$, and the max function. Thus, the result also holds for more specific families of functions considered in the literature, which satisfy additional regularity properties, as long as they contain any Lipschitz function composed as above (for example, Hadamard semi-differentiable functions in \citet{zhang2020complexity}, Whitney-stratifiable functions in \citet{davis2018stochastic}, regular functions in \citet{clarke1990optimization}, etc.).
\end{remark}

\begin{remark}[Strengthening of Proposition \ref{prop:stat}]\label{remark:propstrong}
	The proof of \thmref{thm:main} uses a construction that actually strengthens Proposition \ref{prop:stat}: It implies that for any $\delta,\epsilon$ smaller than some constants, there is a Lipschitz, bounded-from-below function on $\reals^d$, such that the origin is $(\delta,0)$-stationary, yet there are no $\epsilon$-stationary points even at a constant distance from the origin. See remark \ref{remark:propstrongdetails} in the proof for details.
\end{remark}

\begin{figure}
	\includegraphics[trim=2cm 0cm 1cm 0cm,clip=true,width=0.5\linewidth]{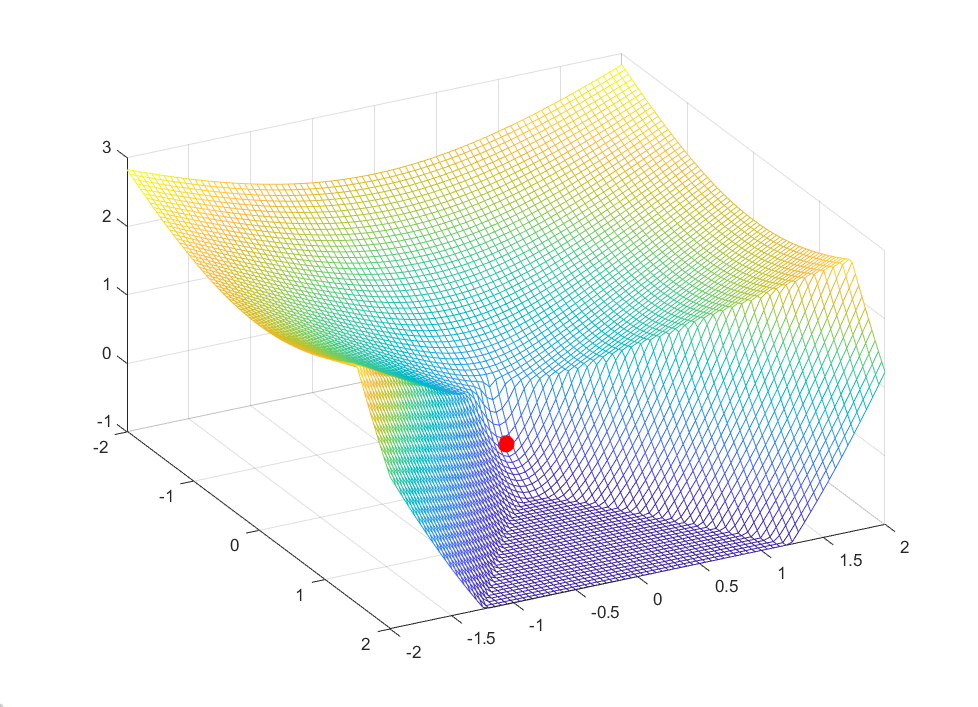}%
	\includegraphics[trim=1cm 0cm 1cm 0cm,clip=true,width=0.5\linewidth]{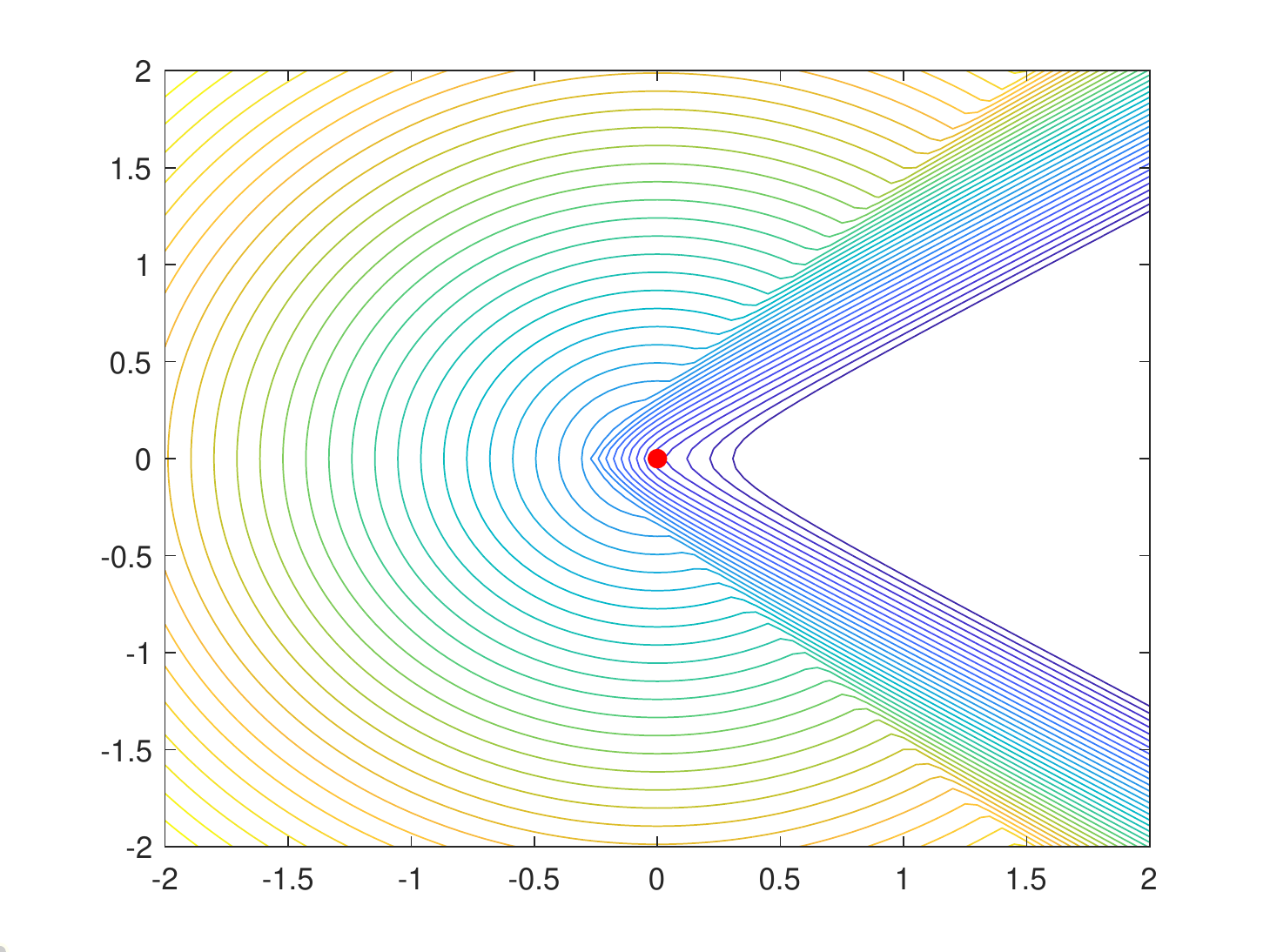}
	\caption{Mesh and Contour plot of the function $\bx\mapsto \max\{-1,g_{\bw}(\bx)\}$ on $\reals^2$, where $\bw=(0.3,0)$ and $g_{\bw}(\cdot)$ is defined in \lemref{lem:channel} (as part of the proof of \thmref{thm:main}). The origin is marked with a red dot. Best viewed in color.}
	\label{fig:channel}
\end{figure}

The formal proof of the theorem appears in the next subsection, but can be informally described as follows: First, we construct a Lipschitz function on $\reals^d$, specified by a small vector $\bw$, which resembles the norm function $\bx\mapsto\norm{\bx}$ in ``most'' of $\reals^d$, but with a ``channel'' leading away from a neighborhood of the origin in the direction of $\bw$, and reaching a completely flat region (see \figref{fig:channel}). We emphasize that the graphical illustration is a bit misleading due to the low dimension: In high dimensions, the ``channel'' and flat region contain a vanishingly small portion of $\reals^d$. This function has the property of having $\epsilon$-stationary points only in the flat region far away from the origin, in the direction of $\bw$, even though the function appears in most places like the norm function $\bx\mapsto \norm{\bx}$ independent of $\bw$. As a result, any oracle-based algorithm, that doesn't happen to hit the vanishingly small region where the function differs from $\bx\mapsto\norm{\bx}$, receives no information about $\bw$, and thus cannot determine where the $\epsilon$-stationary points lie. As a result, such an algorithm cannot return near-approximately-stationary points. 

Unfortunately, the construction described so far does not work as-is, since the algorithm can always query sufficiently close to the origin (closer than roughly $\norm{\bw}$), where the gradients do provide information about $\bw$. To prevent this, we compose the function with an algorithm-dependent affine mapping, which doesn't significantly change the function's properties, but ensures that the algorithm can never get too close to the (mapped) origin. We show that such an affine mapping must exist, based on standard oracle complexity results, which imply that oracle-based algorithms as above cannot get too close to the minimum of a generic convex quadratic function with a bounded number of queries. Using this mapping, and the useful property that $\bw$ can be chosen arbitrarily small, leads to our theorem. The full details appear in the following subsection.

\begin{remark}[Extension to higher-order algorithms]
Our proof approach is quite flexible, in the sense that for any algorithm, we really only need some function  which cannot be optimized to arbitrarily high accuracy, composed with a ``channel'' construction as above. Since functions of this type also exist for algorithms employing higher-order derivatives beyond gradients \citep{arjevani2019oracle}, it is unlikely that such higher-order algorithms will circumvent our impossibility result.
\end{remark}

\subsection{Proof of \thmref{thm:main}}\label{subsec:proof}

We begin by stating the following theorem, which follows from well-known results in oracle complexity (see \citet{nesterov2018lectures,nemirovskyyudin1983}):
\begin{theorem}\label{thm:hardquad}
	For any $T>1$, any algorithm in $\Acal_{det}\cup \Acal_{span}$ and any dimension $d\geq 2T$, there is a vector $\bx^*\in \reals^d$ (where $\norm{\bx^*}\leq \frac{1}{2}$) and a positive definite matrix $M\in \reals^{d\times d}$ (with minimal and maximal eigenvalues satisfying $\frac{1}{2}\leq \lambda_{\min}(M)\leq \lambda_{\max}(M)\leq 1$), such that the iterates $\bx_1,\ldots,\bx_T$ produced by the algorithm  when ran on the strictly convex quadratic function $f(\bx):=(\bx-\bx^*)^\top M(\bx-\bx^*)$ satisfy 
	\[
	\min_{t\in \{1,\ldots,T\}}\norm{\bx_t-\bx^*}~\geq~\exp(-T)~.
	\]
\end{theorem}
For completeness, we provide a self-contained proof in Appendix \ref{app:oracle}. 
Basically, the theorem states that for any algorithm in $\Acal_{det}\cup \Acal_{span}$, there is a relatively well-conditioned\footnote{In the sense that $\lambda_{\max}(M)/\lambda_{\min}(M)\leq 2$.} but still ``hard'' strictly convex quadratic function, whose minimum cannot be detected with accuracy better than $\exp(-\Omega(T))$. 

\begin{remark}[Extension to any gradient-based algorithm]\label{remark:anyalg}
	Up to the constants, a lower bound as in \thmref{thm:hardquad} is widely considered to hold (with high-probability) for \emph{all} oracle-based algorithms, not just for deterministic or linear-span algorithms (see  \citep{nemirovskyyudin1983,simchowitz2018randomized}). In that case, our \thmref{thm:main} can be easily extended to apply to all oracle-based algorithms which utilize function values and gradients, since the only point in the proof where we really need to restrict the algorithm class is in \thmref{thm:hardquad}. Unfortunately, we are not aware of a result in the literature which quite states this, explicitly and in the required form. For example, there are algorithm-independent lower bounds which rely on non-quadratic functions \citep{woodworth2017lower}, or apply to quadratics, but not in a regime where $\lambda_{\max}(M)/\lambda_{\min}(M)$ is a constant as in our case \citep{simchowitz2018randomized}. 
\end{remark}

Given the theorem, our first step will be to reduce it to a hardness result for optimizing convex Lipschitz functions of the form $\bx\mapsto \norm{M^{1/2}(\bx-\bx^*)}$:
\begin{lemma}\label{lem:hardf}
	For any algorithm in $\Acal_{det}\cup \Acal_{span}$, any $T>1$ and any dimension $d\geq 2T$, there is a vector $\bx^*\in \reals^d$ (where $\norm{\bx^*}\leq \frac{1}{2}$) and a positive definite matrix $M\in \reals^{d\times d}$ (with $\frac{1}{2}\leq \lambda_{\min}(M)\leq \lambda_{\max}(M)\leq 1$), such that the convex function 
	\[
	\tilde{f}(\bx)~:=~\norm{M^{1/2}(\bx-\bx^*)}
	\] 
	satisfies the following:
	\begin{itemize}
		\item $\tilde{f}(\cdot)$ is $\frac{1}{\sqrt{2}}$-Lipschitz, and $\tilde{f}(\mathbf{0})\leq \frac{1}{2}$.
		\item If we run the algorithm on $\tilde{f}(\cdot)$, then $\min_{t\in \{1,\ldots,T\}}\norm{\bx_t-\bx^*}\geq \exp(-T)$.
	\end{itemize}
\end{lemma}
\begin{proof}
	We will start with the second bullet. Fix some algorithm $A$ in $\Acal_{det}\cup \Acal_{span}$, and assume by contradiction that for any $\bx^*,M$ satisfying the conditions in the lemma, the algorithm runs on $\tilde{f}(\cdot)$ and produces iterates such that $\min_{t\in \{1,\ldots,T\}}\norm{\bx_t-\bx^*}<\exp(-T)$ (either deterministically if the algorithm is deterministic, or for some realization of its random coin flips, if it is randomized). But then, we argue that given access to gradients and values of $f(\bx):=\tilde{f}^2(\bx)=(\bx-\bx^*)^\top M (\bx-\bx^*)$, we can use $A$ to specify \emph{another} algorithm in $\Acal_{det}\cup \Acal_{span}$ that runs on $f(\cdot)$ and produces points $\bx_1,\ldots,\bx_T$ such that $\min_{t\in \{1,\ldots,T\}}\norm{\bx_t-\bx^*}<\exp(-T)$, contradicting \thmref{thm:hardquad}. To see why, note that given access to an oracle returning values and gradients of $f(\cdot)$ at $\bx$, we can simulate an oracle returning gradients and values of $\tilde{f}(\cdot)$ at $\bx$ via the easily-verified formulaes
	\[
	\tilde{f}(\bx)=\sqrt{f(\bx)}~~~\text{and}~~~\nabla \tilde{f}(\bx)= \frac{1}{2\sqrt{f(\bx)}}\nabla f(\bx)
	\]
	(and for $\bx=\bx^*$ where $\tilde{f}(\cdot)$ is not differentiable, we can just return the value $0$ and the generalized gradient $\mathbf{0}$). We then feed the responses of this simulated oracle to $A$, and get the resulting $\bx_1,\ldots,\bx_T$. This give us a new algorithm $A'$, which is easily verified to be in $\Acal_{det}\cup \Acal_{span}$ if the original algorithm $A$ is in $\Acal_{det}\cup \Acal_{span}$. 
	
	It remains to prove the second bullet in the lemma. First, we have $\tilde{f}(\mathbf{0})=\norm{M^{1/2}\bx^*}\leq \sqrt{\norm{M}}\norm{\bx^*}\leq \frac{1}{2}$. Second, we note that for any $\bx\neq \bx^*$, $\tilde{f}(\cdot)$ is differentiable and satisfies
	\[
	\norm{\nabla \tilde{f}(\bx)}~=~
	\frac{\norm{M(\bx-\bx^*)}}{2\norm{M^{1/2}(\bx-\bx^*)}}
	~\leq~ \frac{\lambda_{\max}(M)\cdot \norm{\bx-\bx^*}}{2\sqrt{\lambda_{\min}(M)}\cdot \norm{\bx-\bx^*}}~=~
	\frac{\lambda_{\max}(M)}{2\sqrt{\lambda_{\min}(M)}}~,
	\] 
	which by the conditions on $M$, is at most $\frac{1}{2\sqrt{1/2}}=\frac{1}{\sqrt{2}}$.
\end{proof}

Next, we define a function $g(\cdot)$ with two properties: It is identical to $\bx\mapsto\norm{\bx}$ in parts of $\reals^d$ (in fact, as we will see later, in ``almost'' all of $\reals^d$), yet unlike the function $\bx\mapsto \norm{\bx}$, it has no stationary points, or even $\epsilon$-stationary points. 
\begin{lemma}\label{lem:channel}
	Fix some vector $\bw\neq \mathbf{0}$ in $\reals^d$, and define the function
	\[
	g_{\bw}(\bx) ~:=~ \norm{\bx}-\left[4\bar{\bw}^\top (\bx+\bw)-2\norm{\bx+\bw}\right]_+,
	\]
	where $\bar{\bu}:=\bu/\norm{\bu}$ for any vector $\bu$, and $[v]_+:=\max\{v,0\}$. Then $g_{\bw}(\cdot)$ is $7$-Lipschitz, and has no $\epsilon$-stationary points for any $\epsilon<\frac{1}{\sqrt{2}}$~. 
\end{lemma}
\begin{proof}
	In the proof, we will drop the $\bw$ subscript and refer to $g_{\bw}(\cdot)$ as $g(\cdot)$.
	
	The functions $\bx\mapsto \norm{\bx}$, $\bx\mapsto 4\bar{\bw}^\top (\bx+\bw)$, $\bx\mapsto2 \norm{\bx+\bw}$ and $x\mapsto \max\{0,x\}$ are respectively $1$-Lipschitz, $4$-Lipschitz, $2$-Lipschitz and $1$-Lipschitz, from which it immediately follows that $g(\cdot)$ is $1+4+2=7$ Lipschitz. Thus, it only remains to show that $g(\cdot)$ has no $\epsilon$-stationary points. 
	
	It is easily seen that the function $g(\cdot)$ is not differentiable at only $3$ possible regions: (1) $\bx=\mathbf{0}$, (2) $\bx=-\bw$, and (3) $\{\bx:4\bar{\bw}^\top (\bx+\bw)-2\norm{\bx+\bw}=0\}$ (or equivalently, $\{\bx:\bar{\bw}^\top(\overline{\bx+\bw})=\frac{1}{2}\}$ if we exclude $\bx=-\bw$), which are all measure-zero sets in $\reals^d$. At any other point, $g(\cdot)$ is differentiable and the gradient satisfies
	\[
	\nabla g(\bx) = \bar{\bx}-\mathbf{1}_{\bar{\bw}^\top(\overline{\bx+\bw})> \frac{1}{2}}\cdot\left(4\bar{\bw}-2(\overline{\bx+\bw})\right)~.
	\]
	Moreover, at those differentiable points, if $\bar{\bw}^\top(\overline{\bx+\bw})< \frac{1}{2}$ then
	\[
	\norm{\nabla g(\bx)}~=~\norm{\bar{\bx}}~=~1~,
	\] 
	and if $\bar{\bw}^\top(\overline{\bx+\bw})> \frac{1}{2}$, then by the triangle inequality,
	\begin{align*}
	\norm{\nabla g(\bx)}~&=~\norm{\bar{\bx}-\left(4\bar{\bw}-2(\overline{\bx+\bw})\right)}
	~=~\norm{4\bar{\bw}-2(\overline{\bx+\bw})-\bar{\bx}}\\
	&\geq~
	4\norm{\bar{\bw}}-2\norm{\overline{\bx+\bw}}-\norm{\bar{\bx}}~=~
	4-2-1 = 1~.
	\end{align*}
	Thus, no differentiable point of $g$ is even $0.99$-stationary. It remains to show that even the non-differentiable points of $g$ are not $\epsilon$-stationary for any $\epsilon < \frac{1}{\sqrt{2}}$. To do so, we will use the facts that $\partial(g_1+g_2)\subseteq \partial g_1 + \partial g_2$, and that if $g_1$ is univariate, $\partial (g_1\circ g_2)(\bx)\subseteq \text{conv}\{r_1 \br_2:r_1\in \partial g_1(g_2(\bx)),\br_2\in \partial g_2(\bx)\}$ (see \citet{clarke1990optimization}). 
	\begin{itemize}
		\item At $\bx=\mathbf{0}$, we have 
		\[
		\partial g(\bx)~\subseteq~ \text{conv}\{\bu-2\bar{\bw}~:~\norm{\bu}\leq 1\}~=~
		\{\bu-2\bar{\bw}~:~\norm{\bu}\leq 1\}~.
		\]
		Any element in this set has a norm of $\norm{\bu-2\bar{\bw}}=\norm{2\bar{\bw}-\bu}\geq 2\norm{\bar{\bw}}-\norm{\bu}\geq 2-1=1$ by the triangle inequality. Thus, $\bx=0$ is not $\epsilon$-stationary for any $\epsilon<1$.
		\item At $\bx=-\bw$, we have
		\begin{align*}
		\partial g(\bx)~&\subseteq~ \text{conv}\{-\bar{\bw}-v\cdot(4\bar{\bw}-2\bu):v\in [0,1],\norm{\bu}\leq 1\}\\
		&=~ \text{conv}\{2v\bu-(1+4v)\bar{\bw}:v\in [0,1],\norm{\bu}\leq 1\}~.
		\end{align*}
		For any element in the set $\{2v\bu-(1+4v)\bar{\bw}:v\in [0,1],\norm{\bu}\leq 1\}$ (corresponding to some $v,\bu$), its inner product with $-\bar{\bw}$ is
		\[
		-2v\bar{\bw}^\top\bu+(1+4v)\geq -2v+(1+4v)\geq 1~.
		\]
		Thus, any element in the convex hull of this set, which contains $\partial g(\bx)$, has an inner product of at least $1$ with $-\bar{\bw}$. Since $-\bar{\bw}$ is a unit vector, it follows that the norm of any element in $\partial g(\bx)$ is at least $1$, so this point is not $\epsilon$-stationary for any $\epsilon<1$. 
		\item At any $\bx$ in the set $\{\bx:\bar{\bw}^\top(\overline{\bx+\bw})=\frac{1}{2}\}\setminus\{\mathbf{0},-\bw\}$, we have
		\begin{align}
		\partial g(\bx) ~&\subseteq~ \text{conv}\left\{\bar{\bx}-v\cdot\left(4\bar{\bw}-2(\overline{\bx+\bw})\right)~:~v\in [0,1]\right\}\notag\\
		&=~\left\{\bar{\bx}-v\cdot\left(4\bar{\bw}-2(\overline{\bx+\bw})\right)~:~v\in [0,1]\right\}\label{eq:dg0}\\
		&=~
		\left\{\left(\frac{1}{\norm{\bx}}+\frac{2v}{\norm{\bx+\bw}}\right)\bx
		-\left(\frac{4v}{\norm{\bw}}-\frac{2v}{\norm{\bx+\bw}}\right)\bw~:~v\in [0,1]\right\}~.\label{eq:dg}
		\end{align}
		Let $\bx=\bx_{|}+\bx_{\perp}$, where $\bx_{\perp} = (I-\bar{\bw}\bar{\bw}^\top)\bx$ is the component of $\bx$ orthogonal to $\bw$, and $\bx_|\in \text{span}(\bw)$. Thus, any element in $\partial g(\bx)$ can be written as
		\[
		\left(\frac{1}{\norm{\bx}}+\frac{2v}{\norm{\bx+\bw}}\right)\bx_{\perp}+a\cdot\bw
		\]
		for some scalar $a$. Since $\bw$ is orthogonal to $\bx_{\perp}$, the norm of this element is at least
		\[
		\left(\frac{1}{\norm{\bx}}+\frac{2v}{\norm{\bx+\bw}}\right)\norm{\bx_{\perp}}
		~\geq~
		\frac{1}{\norm{\bx}}\cdot\norm{\bx_{\perp}}~.
		\]
		Noting that
		\[
		\norm{\bx_{\perp}}^2=\bx^\top(I-\bar{\bw}\bar{\bw}^\top)^2 \bx= \bx^\top(I-\bar{\bw}\bar{\bw}^\top) \bx=\norm{\bx}^2-(\bar{\bw}^\top\bx)^2 =
		\norm{\bx}^2(1-(\bar{\bw}^\top\bar{\bx})^2)
		\]
		and plugging into the above, it follows that the norm is at least $\sqrt{(1-(\bar{\bw}^\top\bar{\bx})^2)}$.
		
		Now, let us suppose that there exists an element in $\partial g(\bx)$ with norm at most $\epsilon$. By the above, it follows that
		\begin{equation}
		\sqrt{(1-(\bar{\bw}^\top\bar{\bx})^2)}\leq \epsilon~.\label{eq:wx}
		\end{equation}
		However, we will show that for any $\epsilon<\frac{1}{\sqrt{2}}$, we must arrive at a contradiction, which implies that $\bx$ cannot be $\epsilon$-stationary for $\epsilon<\frac{1}{\sqrt{2}}$. To that end, let us consider two cases:
		\begin{itemize}
			\item If $\bar{\bw}^\top\bar{\bx}\leq 0$, then by \eqref{eq:wx}, we must have $\bar{\bw}^\top\bar{\bx}\leq -\sqrt{1-\epsilon^2}$. But then, for any $\bu\in \partial g(\bx)$, by \eqref{eq:dg0} and our assumption that $\bar{\bw}^\top(\overline{\bx+\bw})=\frac{1}{2}$,
			\[
			\bar{\bw}^\top\bu ~=~ \bar{\bw}^\top\bar{\bx}-v\cdot\left(4-2\cdot\frac{1}{2}\right)
			~\leq~ -\sqrt{1-\epsilon^2}-3v~\leq~ -\sqrt{1-\epsilon^2}.
			\]
			This implies that $\norm{\bu}\geq \sqrt{1-\epsilon^2}$ for any $\bu\in \partial g(\bx)$. Thus, if there was some $\bu\in \partial g(\bx)$ with norm at most $\epsilon$, we get that $\epsilon \geq \sqrt{1-\epsilon^2}$, which cannot hold if $\epsilon < \frac{1}{\sqrt{2}}$.
			\item If $\bar{\bw}^\top\bar{\bx}>0$, then by \eqref{eq:wx}, we have $\bar{\bw}^\top \bar{\bx}\geq \sqrt{1-\epsilon^2}$. Hence,
			\[
			\bar{\bw}^\top(\bx+\bw)~\geq~\norm{\bx}\sqrt{1-\epsilon^2}+\norm{\bw}~\geq~ (\norm{\bx}+\norm{\bw})\sqrt{1-\epsilon^2}~\geq~\norm{\bx+\bw}\sqrt{1-\epsilon^2}~.
			\]
			However, dividing both sides by $\norm{\bx+\bw}$, we get that $\bar{\bw}^\top(\overline{\bx+\bw})~\geq~\sqrt{1-\epsilon^2}$. If $\epsilon<\frac{1}{\sqrt{2}}$, it follows that $\bar{\bw}^\top(\overline{\bx+\bw})> \frac{1}{\sqrt{2}}$, which contradicts our assumption that $\bx$ satisfies $\bar{\bw}^\top (\overline{\bx+\bw})=\frac{1}{2}$. 
		\end{itemize}
	\end{itemize}
\end{proof}

\begin{remark}\label{remark:propstrongdetails}
	The function 
	\[
	\tilde{g}_{\bw}(\bx):=\max\{g_{\bw}(\mathbf{0})-1~,~g_{\bw}(\bx)\}~,
	\] 
	where $g_{\bw}(\cdot)$ is as defined in \lemref{lem:channel}, actually strengthens Proposition \ref{prop:stat} from the introduction: According to the lemma, $g_{\bw}(\cdot)$ is $7$-Lipschitz and has no $\epsilon$-stationary points for $\epsilon < 1/\sqrt{2}$. Therefore, it is easily verified that for any $\bw$, $\tilde{g}_{\bw}(\cdot)$ is $7$-Lipschitz, bounded from below, and any $\epsilon$-stationary point is at a distance of at least $1/7$ from the origin\footnote{The last point follows from the fact that if $\by$ is an $\epsilon$-stationary point of $\tilde{g}_{\bw}(\cdot)$, then we can find an arbitrarily close point $\bx$ such that $\tilde{g}_{\bw}(\bx)\neq g_{\bw}(\bx)$, hence $g_{\bw}(\bx)< g_{\bw}(\mathbf{0})-1$, and as a result $g_{\bw}(\mathbf{0})-g_{\bw}(\bx)>1$. But $g_{\bw}(\cdot)$ is $7$-Lipschitz, hence $\norm{\bx}> 1/7$, and therefore $\norm{\by}\geq 7$.}. However, we also claim that the origin is a $(\delta,0)$-stationary point for any $\delta\in \left(0,1/7\right)$. To see this, note first that for such $\delta$, by the Lipschitz property of $g_{\bw}(\cdot)$, we have $\tilde{g}_{\bw}(\bx)=g_{\bw}(\bx)$ in a $\delta$-neighborhood of the origin. Fix any $\bw$ such that $\norm{\bw}=\frac{\delta}{2}$, and let $\bv$ be any vector of norm $\delta$ orthogonal to $\bw$. It is easily verified that $\bar{\bw}^\top(\overline{\bv+\bw})<\frac{1}{2}$, in which case
\[
\nabla \tilde{g}_{\bw}(\bv)~=~\nabla g_{\bw}(\bv) ~=~ \bar{\bv}~,
\]
and therefore $\frac{1}{2}\left(\nabla \tilde{g}_{\bw}(\bv)+\nabla \tilde{g}_{\bw}(-\bv)\right)=\mathbf{0}$. 
\end{remark}

\begin{lemma}
	Fix any algorithm in $\Acal_{det}\cup \Acal_{span}$, any $T>1$ and any $d\geq 2T$. Define the function
	\begin{align*}
	h_{\bw}(\bx)~&:=~ \max\{-1~,~g_{\bw}(M^{1/2}(\bx-\bx^*))\}\\
	&=~ \max\left\{-1,\norm{M^{1/2}(\bx-\bx^*)}-\left[4\bar{\bw}^\top (M^{1/2}(\bx-\bx^*)+\bw)-2\norm{M^{1/2}(\bx-\bx^*)+\bw}\right]_+\right\}~,
	\end{align*}
	where $M,\bw^*$ are as defined in \lemref{lem:hardf}, $g_{\bw}(\cdot)$ is as defined in \lemref{lem:channel}, and $\bw$ is a vector of norm $\frac{1}{300}\exp(-T)$ in $\reals^d$. Then:
	\begin{itemize}
		\item $h_{\bw}(\cdot)$ is $7$-Lipschitz, and satisfies  $h_{\bw}(\mathbf{0})-\inf_{\bx}h_{\bw}(\bx)\leq \frac{3}{2}$.
		\item Any $\epsilon$-stationary point $\bx$ of $h_{\bw}(\cdot)$ for $\epsilon<\frac{1}{2\sqrt{2}}$ satisfies $h_{\bw}(\bx)=-1$. 
		\item There exists a choice of $\bw$, such that if we run the algorithm on $h_{\bw}(\cdot)$, then with probability at least $1-T\exp(-d/18)$ over the algorithm's randomness (or deterministically if the algorithm is deterministic), the algorithm's iterates $\bx_1,\ldots,\bx_T$ satisfy
		$\min_{t\in \{1,\ldots,T\}}h_{\bw}(\bx_t)>0$.
	\end{itemize}
\end{lemma}
\begin{proof}
	The Lipschitz bound follows from the facts that $z\mapsto \max\{-1,z\}$ is $1$-Lipschitz, $\bx\mapsto M^{1/2}(\bx-\bx^*)$ is $\norm{M^{1/2}}\leq 1$-Lipschitz, and that $g_{\bw}$ is $7$-Lipschitz by \lemref{lem:channel}. Moreover, we clearly have $\inf_{\bx}h_{\bw}(\bx)\geq -1$, and by definition of $h_{\bw}(\cdot)$ and \lemref{lem:hardf},
	\[
	h_{\bw}(\mathbf{0})~\leq~ \norm{-M^{1/2}\bx^*}=\tilde{f}(\mathbf{0})\leq \frac{1}{2}~.
	\]
	Combining the two observations establishes the first bullet in the lemma.
	
	As to the second bullet, let $\tilde{g}(\bx):=g_{\bw}(M^{1/2}(\bx-\bx^*))$ (so that $h_{\bw}(\bx)=\max\{-1,\tilde{g}(\bx)\}$). It is easily verified that $\bu\in\partial g_{\bw}(\bx)$ if and only if $M^{1/2}\bu\in \partial \tilde{g}(\bx+\bx^*)$. By \lemref{lem:channel}, $g_{\bw}$ has no $\epsilon$-stationary point for $\epsilon<\frac{1}{\sqrt{2}}$, which implies that $\tilde{g}(\bx)$ has no $\epsilon$-stationary points for any $\epsilon$ less than $\lambda_{\min}(M^{1/2})\frac{1}{\sqrt{2}}\geq \frac{1}{2\sqrt{2}}$. But since $h_{\bw}(\bx)=\max\{-1,\tilde{g}(\bx)\}$, it follows that any $\epsilon$-stationary points of $h_{\bw}(\cdot)$ must be arbitrarily close to the region where $h_{\bw}(\cdot)$ is different than $\tilde{g}(\cdot)$, namely where it takes a value of $-1$. Since $h_{\bw}(\cdot)$ is Lipschitz, it follows that its value is $-1$ at the $\epsilon$-stationary point as well.
	
	We now turn to establish the third bullet in the lemma.	
	A crucial observation here is that 
	\begin{equation}\label{eq:almosthard}
	h_{\bw}(\bx) ~=~ g_{\bw}(M^{1/2}(\bx-\bx^*)) ~=~ \tilde{f}(\bx)~~~~\forall \bx:\bar{\bw}^\top\left(\overline{M^{1/2}(\bx-\bx^*)+\bw}\right)\leq\frac{1}{2}~,
	\end{equation}
	where $\tilde{f}(\bx)=\norm{M^{1/2}(\bx-\bx^*)}$ is the ``hard function'' defined in \lemref{lem:hardf}\footnote{Also, the equation can be verified to hold in the corner case where $M^{1/2}(\bx-\bx^*)+\bw=\mathbf{0}$, in which the condition in \eqref{eq:almosthard} is undefined.}. To see why, note first that by definition of $g_{\bw}(\cdot)$ in \lemref{lem:channel}, for any $\bx$ which satisfies the condition in the displayed equation above, we have $g_{\bw}(M^{1/2}(\bx-\bx^*))=\norm{M^{1/2}(\bx-\bx^*)}=\tilde{f}(\bx)$. On the other hand, since this is a non-negative function, it follows that it also equals $\max\{-1,g_{\bw}(M^{1/2}(\bx-\bx^*))\}=h_{\bw}(\bx)$ for such $\bx$, establishing the displayed equation above. 
	
	Next, we will show that \eqref{eq:almosthard} also holds over a set of $\bx$'s which have a more convenient form. To do so, fix some $\bx$ which satisfies the \emph{opposite} condition $\bar{\bw}^\top\left(\overline{M^{1/2}(\bx-\bx^*)+\bw}\right)>\frac{1}{2}$. Then multiplying both sides by $\norm{M^{1/2}(\bx-\bx^*)+\bw}$, we get
	\begin{align*}
	\bar{\bw}^\top M^{1/2}(\bx-\bx^*)+\bar{\bw}^\top \bw ~>~
	\frac{1}{2}\norm{M^{1/2}(\bx-\bx^*)+\bw}~\geq~
	\frac{1}{2}\left(\norm{M^{1/2}(\bx-\bx^*)}-\norm{\bw}\right)~.
	\end{align*}
	Since $\frac{1}{2}\leq \lambda_{\min}(M)\leq\lambda_{\max}(M)\leq 1$ by \lemref{lem:hardf}, it follows that
	\[
	\bar{\bw}^\top M^{1/2}(\bx-\bx^*)+\norm{\bw}~>~ \frac{1}{2}\left(\frac{1}{\sqrt{2}}\norm{\bx-\bx^*}-\norm{\bw}\right)~.
	\]
	For $\bx=\bx^*$, the condition above is trivially satisfied. For $\bx\neq\bx^*$, dividing both sides by $\norm{M^{1/2}(\bx-\bx^*)}$ (which is between $\norm{\bx-\bx^*}$ and $\frac{1}{2}\norm{\bx-\bx^*}$) and simplifying a bit, we get that
	\[
	\bar{\bw}^\top \left(\overline{M^{1/2}(\bx-\bx^*)}\right)~>~
	\frac{1}{2\sqrt{2}}-\frac{3\norm{\bw}}{2\cdot \frac{1}{2}\norm{\bx-\bx^*}}~=~
	\frac{1}{2\sqrt{2}}-\frac{\exp(-T)}{100\norm{\bx-\bx^*)}}~.
	\]
	Noting that any $\bx$ which does not satisfy the condition in  \eqref{eq:almosthard} satisfy the condition above, we get that \eqref{eq:almosthard} implies 
	\begin{equation}\label{eq:almosthard2}
	h_{\bw}(\bx)=\tilde{f}(\bx)=\norm{M^{1/2}(\bx-\bx^*)}~~~\forall \bx\neq \bx^*~\text{s.t.}~\bar{\bw}^\top \left(\overline{M^{1/2}(\bx-\bx^*)}\right)~\leq~\frac{1}{2\sqrt{2}}-\frac{\exp(-T)}{100\norm{\bx-\bx^*}}~.
	\end{equation}
	
	With this equation in hand, let us first establish the third bullet of the lemma, assuming the algorithm we consider is  deterministic. In order to do so, let $\bx_1^{\tilde{f}},\ldots,\bx_T^{\tilde{f}}$ be the (fixed) iterates produced by the algorithm when ran on $\tilde{f}(\cdot)$, and choose $\bw$ in $h_{\bw}(\cdot)$ to be any vector orthogonal to $\{M^{1/2}(\bx_t^{\tilde{f}}-\bx^*)\}_{t=1}^{T}$ (which is possible since the dimension $d$ is larger than $T$). By \lemref{lem:hardf}, we know that for all $t$, $\norm{\bx_t^{\tilde{f}}-\bx^*}\geq \exp(-T)$, in which case we have
	\[
	\bar{\bw}^\top\left(\overline{M^{1/2}(\bx_t^{\tilde{f}}-\bx^*)}\right)~=~0~<~ \frac{1}{2\sqrt{2}}-\frac{\exp(-T)}{100\exp(-T)}
	~\leq~ \frac{1}{2\sqrt{2}}-\frac{\exp(-T)}{100\norm{\bx_t^{\tilde{f}}-\bx^*}}~.
	\]
	Thus, $\bx_t^{\tilde{f}}$ satisfies the condition in \eqref{eq:almosthard2}, and as a result, $h_{\bw}(\bx_t^{\tilde{f}})=\tilde{f}(\bx_t^{\tilde{f}})$ for all $t$. Moreover, using the fact that $\bx_t^{\tilde{f}}$ is bounded away from $\bx^*$, it is easily verified that the condition in \eqref{eq:almosthard2} also holds for $\bx$ in a small local neighborhood of $\bx_t^{\tilde{f}}$, so actually $h_{\bw}(\cdot)$ is identical to $\tilde{f}(\cdot)$ on these local neigborhoods, implying the same values \emph{and gradients} at $\bx_t^{\tilde{f}}$. As a result, if we run the algorithm on $h_{\bw}(\cdot)$ rather than $f(\cdot)$, then the iterates $\bx_1,\ldots,\bx_T$ produced are identical to $\bx_1^{\tilde{f}},\ldots,\bx_T^{\tilde{f}}$. Since $\norm{\bx_t^{\tilde{f}}-\bx^*}>0$, we have $h_{\bw}(\bx_t)=\tilde{f}(\bx_t^{\tilde{f}})=\norm{M^{1/2}(\bx_t^{\tilde{f}}-\bx^*)}>0$ for all $t$ as required.
	
	We now turn to establish the third bullet of the lemma, assuming the algorithm is randomized. As before, we let $\bx_1^{\tilde{f}},\ldots,\bx_T^{\tilde{f}}$ denote the iterates produced by the algorithm when ran on $\tilde{f}(\cdot)$ (only that now they are possibly random, based on the algorithm's random coin flips). The proof idea is roughly the same, but here the iterates may be randomized, so we cannot choose $\bw$ in some fixed manner. Instead, we will pick $\bw$ independently and uniformly at random among vectors of norm $\frac{1}{300}\exp(-T)$, and show that for any realization of the algorithm's random coin flips, with probability at least $1-T\exp(-d/18)$ over $\bw$, $\min_t h_{\bw}(\bx_t)>0$. This implies that there exists some \emph{fixed} choice of $\bw$, for which $\min_t h_{\bw}(\bx_t)>0$ with the same high probability over the algorithm's randomness, as required\footnote{To see why, assume on the contrary that for any fixed choice of $\bw$, the bad event $\min_t h_{\bw}(\bx_t)\leq 0$ occurs with probability larger than $T\exp(-d/18)$ over the algorithm's randomness. In that case, any randomization over the choice of $\bw$ will still yield $\min_t h_{\bw}(\bx_t)\leq 0$ with probability larger than $T\exp(-d/18)$ over the joint randomness of $\bw$ and the algorithm. In particular, this bad event will hold with probability larger than $T\exp(-d/18)$ for some realization of the algorithm's coin flips.}. To proceed, we collect two observations:
	\begin{enumerate}
	\item By \lemref{lem:hardf}, we know that for any realization of the algorithm's random coin flips, $\min_{t\in \{1,\ldots,T\}} \norm{\bx_t-\bx^*}\geq\exp(-T)>0$.
	\item If we fix some unit vectors $\bu_1,\ldots,\bu_T$ in $\reals^d$, and pick a unit vector $\bu$ uniformly at random, then by a union bound and a standard large deviation bound (e.g., \cite{tkocz2012upper}), $\Pr(\max_t \bu^\top \bu_t\geq a)\leq T\cdot\Pr(\bu^\top \bu_1\geq a)\leq T\exp(-da^2/2)$. Taking $\bw=\bu$,  $\bu_t=\overline{M^{1/2}(\bx_t^{\tilde{f}}-\bx^*)}$ for all $t$ (for some realization of $\bx_t^{\tilde{f}}$), and $a=1/3$, it follows that for any realization of the algorithm's random coin flips, $\max_t \bw^\top (\overline{M^{1/2}\bx_t^{\tilde{f}}-\bx^*})\geq 1/3$ with probability at most $T\exp(-d/18)$ over the choice of $\bw$.
	\end{enumerate}
	Combining the two observations, we get that for any realization of the algorithm's coin flips, with probability at least $1-T\exp(-d/18)$ over the choice of $\bw$, it holds for all $\bx_1^{\tilde{f}},\ldots,\bx_T^{\tilde{f}}$ that
	\[
	\bar{\bw}^\top \left(\overline{M^{1/2}(\bx_t^{\tilde{f}}-\bx^*)}\right)~< \frac{1}{3}~<~\frac{1}{2\sqrt{2}}-\frac{\exp(-T)}{100\exp(-T)}~\leq~\frac{1}{2\sqrt{2}}-\frac{\exp(-T)}{100\norm{\bx-\bx^*}}~,
	\]
	as well as $\norm{\bx_t^{\tilde{f}}-\bx^*}>0$. Using the same argument as in the deterministic case, it follows from \eqref{eq:almosthard2} that $h_{\bw}(\cdot)$ and $\tilde{f}(\cdot)$ coincide in small neighborhoods around $\bx_1^{\tilde{f}},\ldots,\bx_T^{\tilde{f}}$, with probability at least $1-T\exp(-d/18)$. Since the algorithm's iterates depend only on the local values/gradients returned by the oracle, it follows that for any realization of the algorithm's coin flips, with probability at least $1-T\exp(-d/18)$ over the choice of $\bw$, the iterates $\bx_1,\ldots,\bx_T$ and $\bx_1^{\tilde{f}},\ldots,\bx_T^{\tilde{f}}$ are going to be identical, and satisfy
	\[
	\min_t h_{\bw}(\bx_t)=\min_t h_{\bw}(\bx_t^{\tilde{f}})=\min_t \tilde{f}(\bx_t^{\tilde{f}})>0~.
	\]
	This holds for any realization of the algorithm's random coin flips, which as discussed earlier, implies the required result.
\end{proof}

The theorem is now an immediate corollary of the lemma above: With the specified high probability (or deterministically), $\min_t h_{\bw}(\bx_t)>0$, even though all $\epsilon$-stationary points (for any $\epsilon < \frac{1}{2\sqrt{2}}$) have a value of $-1$. Since $h_{\bw}$ is also $7$-Lipschitz, we get that the distance of any $\bx_t$ from an $\epsilon$-stationary point must be at least $\frac{0-(-1)}{7}=\frac{1}{7}$. Simplifying the numerical terms by choosing a large enough constant $C$ and a small enough constant $c$, the theorem follows.

\section{Discussion} \label{sec:discussion}

\thmref{thm:main} implies that at least with black-box oracle access to the function, it is probably impossible to design algorithms with finite-time guarantees for finding near-approximately-stationary points. This raises the question of what alternative notions of stationarity can we consider when trying to efficiently optimize generic nonconvex nonsmooth functions. 

One very appealing notion is the $(\delta,\epsilon)$-stationarity of \cite{zhang2020complexity} that we discussed in the introduction, which comes with clean finite-time guarantees. Our negative result in \thmref{thm:main} provides further motivation to consider it, by showing that a natural strengthening of this notion will not work. However, as we showed in Proposition \ref{prop:stat} and remark \ref{remark:propstrong}, we need to accept that this stationarity notion can have unexpected behavior, and there exist cases where it will not resemble a stationary point in any intuitive sense. 

Another possible direction is to replace the convex hull in the definition of $(\delta,\epsilon)$-stationarity by some fixed convex combination of gradients in the $\delta$-neighborhood of our point. For example, we might define a point $\bx$ as $\widetilde{(\delta,\epsilon)}$-stationary with respect to a function $f(\cdot)$, if $\norm{\E_{\bu}[\nabla f(\bx+\delta \bu)]}\leq \epsilon$, where $\bu$ is uniformly distributed in the unit origin-centered ball. For Lipschitz functions (which are almost everywhere differentiable), this operation is well-defined, and is generally equivalent to finding $\epsilon$-stationary points of the \emph{smoothed function $\tilde{f}(\bx)=\E_{\bu}[f(\bx+\delta \bu)]$}, which can be done efficiently given access to gradients of $f(\cdot)$ (see \citet{duchi2012randomized,ghadimi2013stochastic}). However, it is important to note that the gradient Lipschitz parameter of $\tilde{f}(\cdot)$ is generally dimension-dependent, and thus we will not get dimension-free guarantees if we simply plug in existing results for smooth functions. Moreover, this notion of $\widetilde{(\delta,\epsilon)}$-stationarity still does not rule out counter-intuitive behaviors similar to Proposition \ref{prop:stat}, where a point is $\widetilde{(\delta,0)}$-stationary without actually having a near-zero gradient in its $\delta$-neighborhood. 

\begin{figure}
	\includegraphics[trim=2cm 0cm 1cm 0cm,clip=true,width=0.5\linewidth]{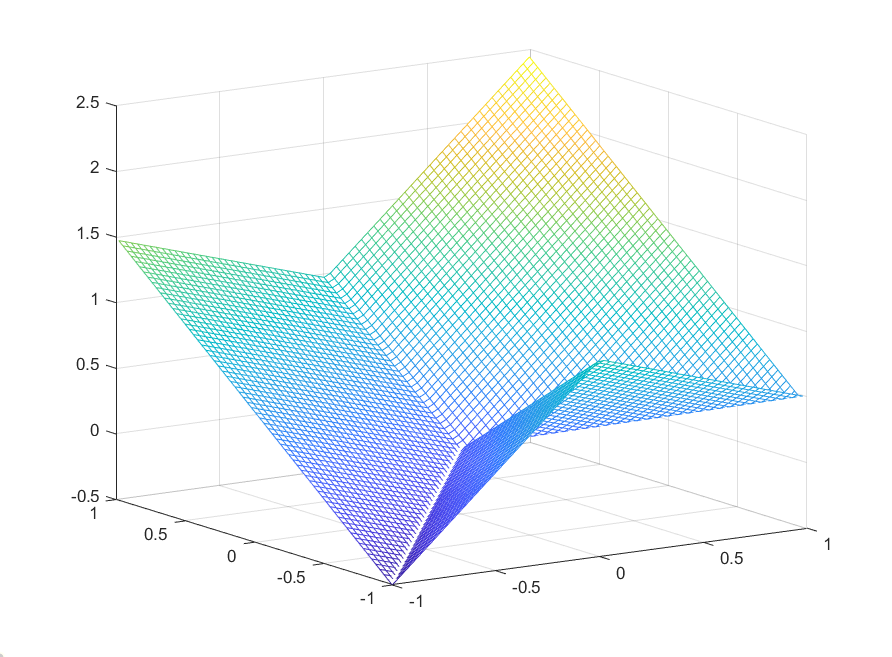}%
	\includegraphics[trim=1cm 0cm 1cm 0cm,clip=true,width=0.5\linewidth]{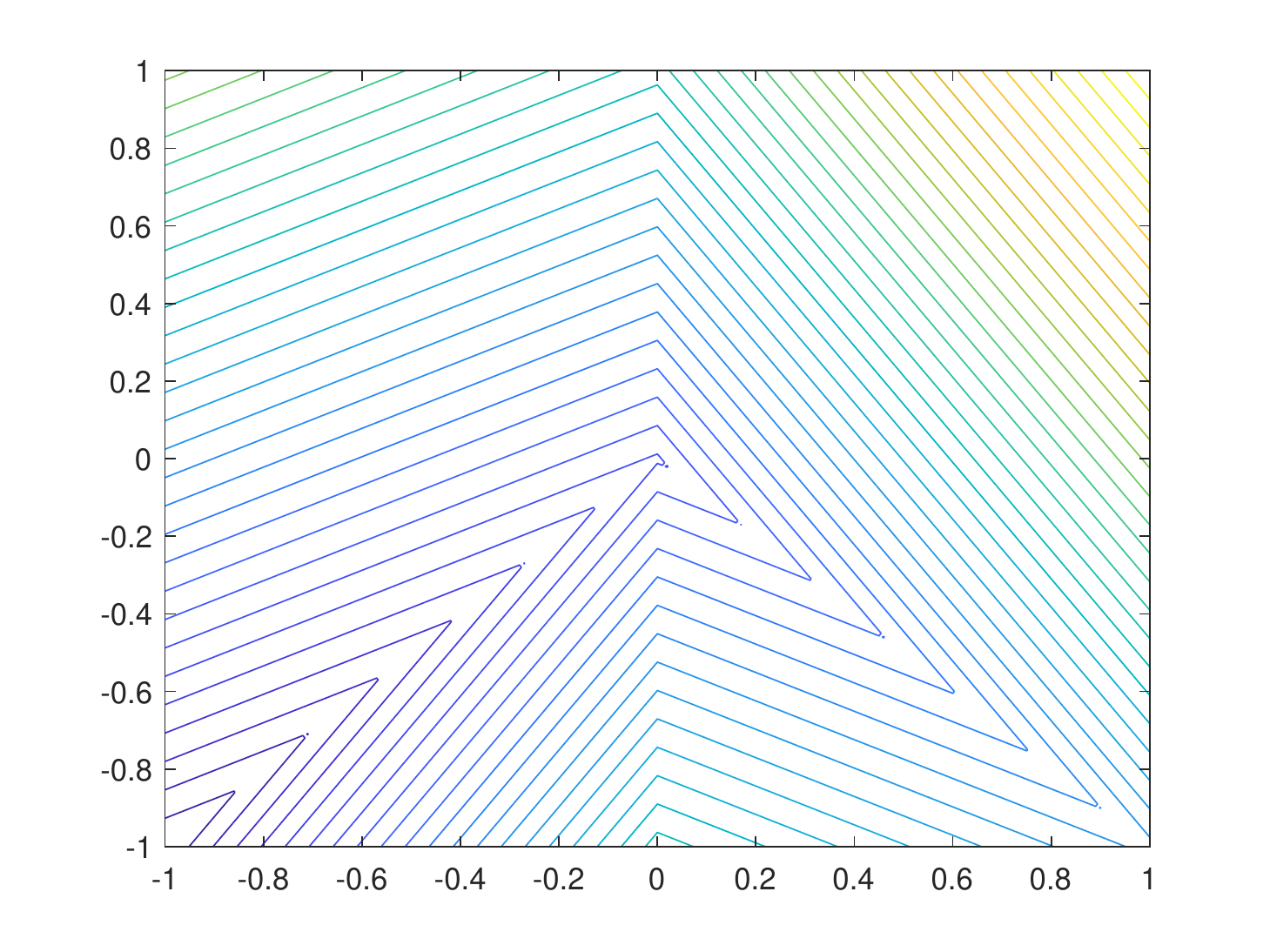}
	\caption{Mesh and Contour plot of the function from \eqref{eq:warga}. Best viewed in color.}
	\label{fig:warga}
\end{figure}

In a related direction, one might consider finding $\epsilon$-stationary points of other smooth approximations of the original function, which have a better behavior. For example, for nonconvex functions, a well-known smoothing operation with dimension-free guarantees is the Lasry-Lions regularization \citep{lasry1986remark,attouch1993approximation}, which is closely related to the Moreau-Yosida regularization for smoothing convex functions. However, this operation does not appear to be efficiently computable in general.

Finally, it is important to step back and point out that when considering optimization of nonsmooth nonconvex functions, it is generally difficult to relate stationarity properties to any meaningful local optimality properties, even more so than in the smooth case. For example, consider the simple nonsmooth bivariate function studied in \citep{warga1981fat,czarnecki2006approximation},
\begin{equation}\label{eq:warga}
f(u,v) = \Big|\left|u\right|+v\Big|+\frac{1}{2}u~,
\end{equation}
which is illustrated in \figref{fig:warga}. It can be shown that the origin is a (Clarke) stationary point. However, it is not even approximately-stationary with respect to standard smooth approximations of the function, regardless of how tight we make them. Also, it is clearly not a point we would like our algorithm to converge to, if we actually try to minimize the function. 
This suggests looking for notions beyond stationarity, for which we can still provide algorithmic guarantees even for nonconvex nonsmooth functions.

\bibliographystyle{plainnat}
\bibliography{bib}

\begin{thebibliography}{28}
\providecommand{\natexlab}[1]{#1}
\providecommand{\url}[1]{\texttt{#1}}
\expandafter\ifx\csname urlstyle\endcsname\relax
  \providecommand{\doi}[1]{doi: #1}\else
  \providecommand{\doi}{doi: \begingroup \urlstyle{rm}\Url}\fi

\bibitem[Arjevani et~al.(2019)Arjevani, Shamir, and Shiff]{arjevani2019oracle}
Yossi Arjevani, Ohad Shamir, and Ron Shiff.
\newblock Oracle complexity of second-order methods for smooth convex
  optimization.
\newblock \emph{Mathematical Programming}, 178\penalty0 (1-2):\penalty0
  327--360, 2019.

\bibitem[Attouch and Aze(1993)]{attouch1993approximation}
H{\'e}dy Attouch and Dominique Aze.
\newblock Approximation and regularization of arbitrary functions in hilbert
  spaces by the lasry-lions method.
\newblock In \emph{Annales de l'Institut Henri Poincare (C) Non Linear
  Analysis}, volume~10, pages 289--312. Elsevier, 1993.

\bibitem[Beck and Hallak(2020)]{beck2020convergence}
Amir Beck and Nadav Hallak.
\newblock On the convergence to stationary points of deterministic and
  randomized feasible descent directions methods.
\newblock \emph{SIAM Journal on Optimization}, 30\penalty0 (1):\penalty0
  56--79, 2020.

\bibitem[Bena{\"\i}m et~al.(2005)Bena{\"\i}m, Hofbauer, and
  Sorin]{benaim2005stochastic}
Michel Bena{\"\i}m, Josef Hofbauer, and Sylvain Sorin.
\newblock Stochastic approximations and differential inclusions.
\newblock \emph{SIAM Journal on Control and Optimization}, 44\penalty0
  (1):\penalty0 328--348, 2005.

\bibitem[Bolte et~al.(2018)Bolte, Sabach, Teboulle, and
  Vaisbourd]{bolte2018first}
J{\'e}r{\^o}me Bolte, Shoham Sabach, Marc Teboulle, and Yakov Vaisbourd.
\newblock First order methods beyond convexity and lipschitz gradient
  continuity with applications to quadratic inverse problems.
\newblock \emph{SIAM Journal on Optimization}, 28\penalty0 (3):\penalty0
  2131--2151, 2018.

\bibitem[Carmon et~al.(2019)Carmon, Duchi, Hinder, and
  Sidford]{carmon2019lower}
Yair Carmon, John~C Duchi, Oliver Hinder, and Aaron Sidford.
\newblock Lower bounds for finding stationary points i.
\newblock \emph{Mathematical Programming}, pages 1--50, 2019.

\bibitem[Clarke(1990)]{clarke1990optimization}
Frank~H Clarke.
\newblock \emph{Optimization and nonsmooth analysis}, volume~5.
\newblock Siam, 1990.

\bibitem[Czarnecki and Rifford(2006)]{czarnecki2006approximation}
Marc-Olivier Czarnecki and Ludovic Rifford.
\newblock Approximation and regularization of lipschitz functions: convergence
  of the gradients.
\newblock \emph{Transactions of the American Mathematical Society},
  358\penalty0 (10):\penalty0 4467--4520, 2006.

\bibitem[Davis and Drusvyatskiy(2019)]{davis2019stochastic}
Damek Davis and Dmitriy Drusvyatskiy.
\newblock Stochastic model-based minimization of weakly convex functions.
\newblock \emph{SIAM Journal on Optimization}, 29\penalty0 (1):\penalty0
  207--239, 2019.

\bibitem[Davis et~al.(2018)Davis, Drusvyatskiy, Kakade, and
  Lee]{davis2018stochastic}
Damek Davis, Dmitriy Drusvyatskiy, Sham Kakade, and Jason~D Lee.
\newblock Stochastic subgradient method converges on tame functions.
\newblock \emph{Foundations of computational mathematics}, pages 1--36, 2018.

\bibitem[Drusvyatskiy and Paquette(2019)]{drusvyatskiy2019efficiency}
Dmitriy Drusvyatskiy and Courtney Paquette.
\newblock Efficiency of minimizing compositions of convex functions and smooth
  maps.
\newblock \emph{Mathematical Programming}, 178\penalty0 (1-2):\penalty0
  503--558, 2019.

\bibitem[Duchi and Ruan(2018)]{duchi2018stochastic}
John~C Duchi and Feng Ruan.
\newblock Stochastic methods for composite and weakly convex optimization
  problems.
\newblock \emph{SIAM Journal on Optimization}, 28\penalty0 (4):\penalty0
  3229--3259, 2018.

\bibitem[Duchi et~al.(2012)Duchi, Bartlett, and
  Wainwright]{duchi2012randomized}
John~C Duchi, Peter~L Bartlett, and Martin~J Wainwright.
\newblock Randomized smoothing for stochastic optimization.
\newblock \emph{SIAM Journal on Optimization}, 22\penalty0 (2):\penalty0
  674--701, 2012.

\bibitem[Ghadimi and Lan(2013)]{ghadimi2013stochastic}
Saeed Ghadimi and Guanghui Lan.
\newblock Stochastic first-and zeroth-order methods for nonconvex stochastic
  programming.
\newblock \emph{SIAM Journal on Optimization}, 23\penalty0 (4):\penalty0
  2341--2368, 2013.

\bibitem[Goldstein(1977)]{goldstein1977optimization}
AA~Goldstein.
\newblock Optimization of lipschitz continuous functions.
\newblock \emph{Mathematical Programming}, 13\penalty0 (1):\penalty0 14--22,
  1977.

\bibitem[Jin et~al.(2017)Jin, Ge, Netrapalli, Kakade, and
  Jordan]{jin2017escape}
Chi Jin, Rong Ge, Praneeth Netrapalli, Sham~M Kakade, and Michael~I Jordan.
\newblock How to escape saddle points efficiently.
\newblock In \emph{Proceedings of the 34th International Conference on Machine
  Learning-Volume 70}, pages 1724--1732. JMLR. org, 2017.

\bibitem[Kiwiel(2007)]{kiwiel2007convergence}
Krzysztof~C Kiwiel.
\newblock Convergence of the gradient sampling algorithm for nonsmooth
  nonconvex optimization.
\newblock \emph{SIAM Journal on Optimization}, 18\penalty0 (2):\penalty0
  379--388, 2007.

\bibitem[Lan and Zhou(2018)]{lan2018optimal}
Guanghui Lan and Yi~Zhou.
\newblock An optimal randomized incremental gradient method.
\newblock \emph{Mathematical programming}, 171\penalty0 (1-2):\penalty0
  167--215, 2018.

\bibitem[Lasry and Lions(1986)]{lasry1986remark}
Jean-Michel Lasry and Pierre-Louis Lions.
\newblock A remark on regularization in hilbert spaces.
\newblock \emph{Israel Journal of Mathematics}, 55\penalty0 (3):\penalty0
  257--266, 1986.

\bibitem[Majewski et~al.(2018)Majewski, Miasojedow, and
  Moulines]{majewski2018analysis}
Szymon Majewski, B{\l}a{\.z}ej Miasojedow, and Eric Moulines.
\newblock Analysis of nonsmooth stochastic approximation: the differential
  inclusion approach.
\newblock \emph{arXiv preprint arXiv:1805.01916}, 2018.

\bibitem[Nemirovsky and Yudin(1983)]{nemirovskyyudin1983}
Arkadii~Semenovich Nemirovsky and David~Borisovich Yudin.
\newblock \emph{Problem complexity and method efficiency in optimization.}
\newblock Wiley, 1983.

\bibitem[Nesterov(2012)]{nesterov2012make}
Yurii Nesterov.
\newblock How to make the gradients small.
\newblock \emph{Optima. Mathematical Optimization Society Newsletter},
  \penalty0 (88):\penalty0 10--11, 2012.

\bibitem[Nesterov(2018)]{nesterov2018lectures}
Yurii Nesterov.
\newblock \emph{Lectures on convex optimization}, volume 137.
\newblock Springer, 2018.

\bibitem[Simchowitz(2018)]{simchowitz2018randomized}
Max Simchowitz.
\newblock On the randomized complexity of minimizing a convex quadratic
  function.
\newblock \emph{arXiv preprint arXiv:1807.09386}, 2018.

\bibitem[Tkocz(2012)]{tkocz2012upper}
Tomasz Tkocz.
\newblock An upper bound for spherical caps.
\newblock \emph{The American Mathematical Monthly}, 119\penalty0 (7):\penalty0
  606--607, 2012.

\bibitem[Warga(1981)]{warga1981fat}
Jack Warga.
\newblock Fat homeomorphisms and unbounded derivate containers.
\newblock \emph{Journal of Mathematical Analysis and Applications}, 81\penalty0
  (2):\penalty0 545--560, 1981.

\bibitem[Woodworth and Srebro(2017)]{woodworth2017lower}
Blake Woodworth and Nathan Srebro.
\newblock Lower bound for randomized first order convex optimization.
\newblock \emph{arXiv preprint arXiv:1709.03594}, 2017.

\bibitem[Zhang et~al.(2020)Zhang, Lin, Sra, and Jadbabaie]{zhang2020complexity}
Jingzhao Zhang, Hongzhou Lin, Suvrit Sra, and Ali Jadbabaie.
\newblock On complexity of finding stationary points of nonsmooth nonconvex
  functions.
\newblock \emph{arXiv preprint arXiv:2002.04130}, 2020.

\end{thebibliography}

\appendix

\section{Proof of \thmref{thm:hardquad}}\label{app:oracle}

In the proof, we let bold-faced letters (e.g., $\bx$) denote vectors, and $x_i$ denote the $i$-th coordinate of the vector $\bx$. Also, we let $\be_1,\be_2,\ldots$ denote the standard basis vectors. 

Our proof will closely follow the analysis employed in \citet[Theorem 3]{lan2018optimal} for a slightly different setting.

Fix an iteration budget $T$ and some dimension $d\geq T$. Let $A$ be the symmetric $d\times d$ tridiagonal matrix defined as
\begin{align*}
\forall 1\leq i < T~~~~~~~~&A(i,i)=2~,~A(i,i+1)=-1\\
\forall 1<i\leq T~~~~~~~~ &A(i,i-1)=-1\\
&A(T,T)=k:=\frac{\sqrt{2}+3}{\sqrt{2}+1}\\
&A(i,j)=0~~~ \text{for all other $(i,j)$}~.
\end{align*}
Also, for some constant $b$ to be determined later, define the quadratic function
\[
g(\bx)~:=~\bx^\top M \bx -\frac{1}{4}\be_1^\top\bx+b~~~\text{where}~~~ M=\frac{1}{8}(A+4I)~.
\]
It is easily verified that this function can be equivalently written as
\begin{equation}\label{eq:chain}
g(\bx)~=~\frac{1}{8}\left(x_1^2+\sum_{i=1}^{T-1}(x_i-x_{i+1})^2+(k-1)x_T^2-2x_1\right)+\frac{1}{2}\norm{\bx}^2+b~.
\end{equation}

We first collect a few useful facts about $g(\cdot)$, stated in the following two lemmas:
\begin{lemma}\label{lem:M}
	$M$ satisfies $\frac{1}{2}\leq\lambda_{\min}(M)\leq\lambda_{\max}(M)\leq 1$. As a result, $M$ is positive definite, and $g(\cdot)$ is strictly convex and has a unique minimum.
\end{lemma}
\begin{proof}
	$A$ is symmetric, and for any $\bx\in \reals^T$, we have
	\[
	\bx^\top A \bx ~=~ x_1^2+\sum_{i=1}^{T-1}(x_i-x_{i+1})^2+(k-1)x_T^2~.
	\]
	This is non-negative, which establishes that $A$ is a positive semidefinite matrix. Hence, by definition of $M$, $\lambda_{\min}(M)=\frac{1}{8}(\lambda_{\min}(A)+4)\geq \frac{1}{8}\cdot 4 = \frac{1}{2}$, which implies that $M$ is positive definite. As a result, $g(\cdot)$ is strictly convex and has a unique minimum. Also, by the displayed equation above,
	\begin{align*}
	\bx^\top A \bx ~&\leq~ x_1^2+2\sum_{i=1}^{T-1}(x_i^2+x_{i+1}^2)+(k-1)x_T^2
	~\leq~ 3x_1^2+\sum_{i=2}^{T-1}(2x_i^2+2x_{i+1}^2)+(k-1)x_T^2\\
	&=~ 3x_1^2+4\sum_{i=2}^{T-1}x_i^2+(k+1)x_T^2~\leq~ 4\norm{\bx}^2~,
	\end{align*}
	where we use the fact that $k\leq 3$. This establishes that $\lambda_{\max}(A)\leq 4$, and therefore $\lambda_{\max}(M)=\frac{1}{8}(\lambda_{\max}(A)+4)\leq 1$.
\end{proof}

\begin{lemma}\label{lem:opt}
	The minimum $\bx^*$ of $g(\cdot)$ is of the form $\bx^*=(q,q^2,\ldots,q^T,0,\ldots,0)$, where $q=\frac{\sqrt{2}-1}{\sqrt{2}+1}$. Moreover, $\norm{\bx^*}\leq \sqrt{\frac{\sqrt{2}-1}{2}}<\frac{1}{2}$.  
\end{lemma}
\begin{proof}
	By the previous lemma and the fact that $g(\cdot)$ is differentiable, $\bx^*$ is the unique point satisfying $\nabla g(\bx^*)=\mathbf{0}$. Thus, it is enough to verify that the formula for $\bx^*$ stated in the lemma indeed satisfies this equation. Computing the gradient of $g(\cdot)$ using the formulation in \eqref{eq:chain}), we just need to verify that
	\[
	6q-q^2-1 = 0~~,~~
	\forall i\in \{2,\ldots,T-1\}~~q^{i-1}-6q^i+q^{i+1}=0
	~~,~~(k+4)q^T-q^{T-1}=0~,
	\]
	or equivalently,
	\[
	1-6q+q^2=0~~,~~ (k+4)q-1=0~,
	\]
	which is easily verified to hold for the value of $q$ stated in the lemma.
	Finally, we have
	\[
	\norm{\bx^*}^2=\sum_{i=1}^{d}(x^*_i)^2=\sum_{i=1}^{T}q^i< \sum_{i=1}^{\infty}q^i~=~\frac{q}{1-q}=\frac{\sqrt{2}-1}{2},
	\] 
	implying $\norm{\bx^*}\leq \sqrt{\frac{\sqrt{2}-1}{2}}$	as required.
\end{proof}

Finally, we assume that the constant term $b$ in \eqref{eq:chain} is fixed so that $g(\bx^*)=0$, which means that $g(\cdot)$ can be written in the form
\begin{equation}\label{eq:fMform}
g(\bx) ~=~ (\bx-\bx^*)^\top M (\bx-\bx^*)~.
\end{equation}

With this construction in hand, we now turn to prove the theorem. We will start with the family of linear-span algorithms $\Acal_{span}$, using any dimension $d\geq T$, and take $g(\cdot)$ as the ``hard'' function on which we will prove a lower bound (note that by the lemmas above and \eqref{eq:fMform}, it satisfies the conditions stated in the theorem).

Consider any algorithm in $\Acal_{span}$, and note that by the structure of $g(\cdot)$ as specified in \eqref{eq:chain}, when the algorithm picks its first query $\bx_1=\mathbf{0}$, it receives a gradient in $\text{span}\{\be_1\}$. Because of the linear-span assumption, it means that $\bx_2\in \text{span}(\be_1)$, which again by \eqref{eq:chain} means that the returned gradient is in $\text{span}\{\be_1,\be_2\}$. Continuing this process, it is easily seen by induction that 
\[
\bx_t\in \text{span}\left\{\be_1,\ldots,\be_{t-1}\right\}
\]
for all $t$, and in particular, $\{\bx_1,\ldots,\bx_T\}\subset \text{span}\left(\be_1,\ldots,\be_{T-1}\right)$. As a result,
\[
\min_{t\in \{1,\ldots,T\}} \norm{\bx_t-\bx^*}^2~\geq~ \inf_{\bx\in \text{span}\{\be_1,\ldots,\be_{T-1}\}} \norm{\bx_t-\bx^*}^2~\geq~
(x_T^*)^2,
\]
which by \lemref{lem:opt} is at least $q^{2T}=\left(\frac{\sqrt{2}-1}{\sqrt{2}+1}\right)^{2T}$. Taking a square root, we get that
\[
\min_{t\in \{1,\ldots,T\}}\norm{\bx_t-\bx^*}~\geq~ \left(\frac{\sqrt{2}-1}{\sqrt{2}+1}\right)^T~\geq~ \exp(-T)~,
\]
as stated in the theorem.

We now turn to prove the theorem for deterministic algorithms. This time, we will let the dimension be any $d\geq 2T$. Fixing an algorithm, and letting $\bu_1,\ldots,\bu_T$ be orthonormal vectors to be specified shortly, we prove the lower bound for the function 
\begin{equation}\label{eq:chaindet}
\tilde{g}(\bx)~:=~\frac{1}{8}\left((\bu_1^\top \bx)^2+\sum_{i=1}^{T-1}(\bu_i^\top \bx-\bu_{i+1}^\top\bx)^2+(k-1)(\bu_T^\top\bx)^2-2\bu_1^\top \bx\right)+\frac{1}{2}\norm{\bx}^2+b~.
\end{equation}
Importantly, we note that
\[
\tilde{g}(\bx)~=~g(U\bx)
\]
where $g(\cdot)$ is the function defined previously in \eqref{eq:chain}, and $U$ is an orthogonal matrix whose first $T$ rows are $\bu_1,\ldots,\bu_T$, and the rest of the rows are some arbitrary completion of the first $T$ rows to an orthonormal basis. Thus, $\tilde{g}(\cdot)$ is equivalent to $g(\cdot)$ up to a rotation of the coordinate system specified by $U$. In particular, using \eqref{eq:fMform}, it follows that
\begin{align*}
\tilde{g}(\bx) ~&=~ (U\bx-\bx^*)^\top M(U\bx-\bw^*) ~=~ (\bx-U^\top \bx^*)^\top (U^\top M U)(\bx-U^\top \bx^*)\\
&=~ (\bx-\tilde{\bx}^*)^\top \tilde{M} (\bx-\tilde{\bx}^*)~,
\end{align*}
where $\tilde{M}=U^\top M U$ and $\tilde{\bx}^*=U^\top \bx^*$. Thus, we see that $\tilde{g}(\cdot)$ has the form required in the theorem, with a matrix $\tilde{M}$ whose spectrum is identical to $M$, and a minimizer $\tilde{\bx}^*=U^\top \bx^*$ whose norm is the same as $\norm{\bx^*}$ (and therefore satisfying the conditions in the theorem). 

We now specify how to choose $\bu_1,\ldots,\bu_T$ in the function definition, so as to get the lower bound on $\min_t \norm{\bx_t-\tilde{\bx}^*}$: Since the algorithm is deterministic, its first query point $\bx_1$ is known in advance. We therefore choose $\bu_1$ to be some unit vector orthogonal to $\bx_1$. Assuming that $\bu_2,\bu_3,\ldots$ are orthogonal to $\{\bu_1,\bx_1\}$ (which we shall justify shortly), we have by \eqref{eq:chaindet} that $\tilde{g}(\bx_1)$ and $\nabla \tilde{g}(\bx_1)$ depend only on $\bx_1,\bu_1$, and not on $\bu_2,\bu_3,\ldots$. As the algorithm is deterministic and depends only on the observed values and gradients, this means that even before choosing $\bu_2,\bu_3,\ldots$, we can already simulate its next iteration, and determine the next query point $\bx_2$. We now pick $\bu_2$ to be some unit vector orthogonal to $\bu_1$ as well as to $\bx_1,\bx_2$. Again by the same considerations, if we assume $\bu_3,\bu_4,\ldots$ are orthogonal to $\{\bu_i,\bx_i\}_{i=1}^{2}$, we have that $\tilde{g}(\bx_2)$ and $\nabla \tilde{g}(\bx_2)$ depend only on $\{\bu_i,\bx_i\}_{i=1}^{2}$, and independent of $\bu_3,\bu_4,\ldots$. So again, we can simulate it and determine the next query point $\bx_3$. We continue this process up to iteration $T$, where we fix $\bu_T$ orthogonal to $\{\bu_i,\bx_i\}_{i=1}^{T-1}$ and to $\bx_T$ (this process is possible as long as the dimension $d$ is at least $2(T-1)+1+1=2T$, as we indeed assume). 

As a result of this process, we get that $\bu_T^\top \bx_t=0$ for all $t\in \{1,\ldots,T\}$. Also, since $\tilde{\bx}^*=U^\top\bx^*$, we also have $\bu_T^\top \tilde{\bx}^* = \bu_T^\top U^\top \tilde{\bx}^* = x_T^*$. Using \lemref{lem:opt}, we get that for all $t\in \{1,\ldots,T\}$,
\[
\norm{\bx_t-\tilde{\bx}^*}^2 \geq (\bu_T^\top(\bx_t-\tilde{\bx}^*))^2 ~=~ (0-x_T^*)^2 = \left(\frac{\sqrt{2}-1}{\sqrt{2}+1}\right)^{2T}~,
\]
which implies that 
\[
\min_{t\in\{1,\ldots,T\}}\norm{\bx_t-\bx^*}~\geq~ \left(\frac{\sqrt{2}-1}{\sqrt{2}+1}\right)^T~\geq~\exp(-T)
\]
as required.

\end{document}